\documentclass[11pt]{article}

\usepackage[margin=2.7cm]{geometry}
\usepackage{amsmath,amssymb,amsthm,amsfonts,mathtools,mathrsfs}
\usepackage[round]{natbib}
\usepackage{xcolor}
\usepackage{booktabs}
\usepackage{microtype}
\usepackage{pgfplots}
\pgfplotsset{compat=1.16}
\usepackage{graphicx}
\definecolor{cHo}{HTML}{E87BA4}\definecolor{cEBf}{HTML}{EB6834}
\definecolor{cEBc}{HTML}{EDA100}\definecolor{cBf}{HTML}{2A78D6}
\definecolor{cBc}{HTML}{1BAF7A}\definecolor{cBen}{HTML}{6250D6}
\definecolor{cNo}{HTML}{8A8A8A}
\usepackage[colorlinks=true,linkcolor=blue!60!black,citecolor=blue!60!black,urlcolor=blue!60!black]{hyperref}

\theoremstyle{plain}
\newtheorem{theorem}{Theorem}
\newtheorem{proposition}{Proposition}
\newtheorem{lemma}{Lemma}
\newtheorem{corollary}{Corollary}
\theoremstyle{definition}
\newtheorem{definition}{Definition}
\newtheorem{assumption}{Assumption}
\theoremstyle{remark}
\newtheorem{remark}{Remark}

\newcommand{\PP}{\mathbb{P}}
\newcommand{\EE}{\mathbb{E}}
\newcommand{\VV}{\operatorname{Var}}
\newcommand{\CC}{\mathcal{C}}
\newcommand{\HH}{\mathcal{H}}
\newcommand{\SSS}{\mathcal{S}}
\newcommand{\EEe}{\mathcal{E}}
\newcommand{\QQ}{\mathcal{Q}}
\newcommand{\ind}{\mathbf{1}}
\newcommand{\osc}{\operatorname{osc}}
\newcommand{\argmax}{\operatorname*{arg\,max}}
\newcommand{\eps}{\varepsilon}
\newcommand{\wh}[1]{\widehat{#1}}
\newcommand{\ov}[1]{\overline{#1}}

\title{Finite-sample certification of coverage regret for budgeted document selection under adaptive stratified auditing}
\author{Nora Serdyukova\\[0.35em]
\small Department of Statistics, University of Concepcion, Chile\\
\small \texttt{nserdyukova@udec.cl}}
\date{}

\begin{document}
\maketitle

\begin{abstract}
Selecting a budget-constrained set of documentary fragments, for instance
the knowledge base of a retrieval-augmented generation system, raises a
statistical question: how much expected query coverage is lost relative to
the best feasible selection when the query distribution is unknown? We
construct finite-sample, distribution-free upper confidence bounds for this
coverage regret that remain valid after data-dependent selection. A pilot
stage may select the document set, screen competitors, build comparison
networks and learn an annotation policy. Then an independent audit acquires
relevance labels with probabilities that are constant within observable
strata. The main result is a simultaneous bound for all retained coverage
contrasts. The bound is made observable by empirical-Bernstein variance calibration or by a
one-stage construction without calibration, and it accounts for estimated
stratum proportions. For maximum-type coverage we prove variance contractions
for swaps of redundant fragments, give a condition under which multiscale
radii improve on a direct union bound, and show that such redundancy cannot
be certified from samples alone. A simulation study confirms validity and
quantifies the gains from variance calibration and certified redundancy.
\end{abstract}

\noindent\textbf{MSC2020 subject classifications:} Primary 62G15, 60E15; secondary 68P20, 90C27.


\medskip

%% =====================================================================
\section{Introduction}\label{sec:intro}
%% =====================================================================

Constructing a compact knowledge base from a large collection of documentary
fragments, for example the retrieval corpus of a retrieval-augmented generation
(RAG) system, requires choosing which fragments to retain under a storage or
annotation budget. The selected collection should support future queries,
whose probability distribution is generally not known in advance. This leads
to a stochastic set-optimization problem: maximize expected query coverage,
then certify the loss in coverage relative to the best feasible collection.
The distinction between a computational approximation guarantee and a
statistical coverage certificate is essential. A near-optimal value for the
empirical selection objective need not imply a small population regret.

For coverage defined by maxima of nonnegative relevance scores, the population
objective is monotone and submodular. This connects the selection problem to
classical submodular maximization \citep{NemhauserWolseyFisher78,KrauseGolovin14}.
Yet the statistical component is different: relevance assessments may be costly,
and queries may be deliberately oversampled in contexts believed to distinguish
competing document selections. Such adaptive data acquisition must not be
allowed to invalidate a confidence statement about the unknown query population.

Concentration for means with unknown variances is classical
\citep{MaurerPontil09,BoucheronLugosiMassart13}. Stratified empirical-Bernstein
methods have been developed for other estimation problems, for instance by
\citet{BurgessChapman21}, and context-dependent evaluation and inverse-probability
correction have substantial precedents \citep{DudikLangfordLi11}.
Generic chaining \citep{Talagrand05,Dirksen15} and local complexity theory
\citep{BartlettBousquetMendelson05,Koltchinskii06} explain why it may be possible
to replace a crude logarithm of the number of candidates by a multiscale
complexity. Consequently, neither stratification nor chaining is claimed to
be new here. The goal is to give a self-contained and precisely scoped
concentration argument for \emph{documentary coverage contrasts}, including
all the conditioning, sampling, and calibration steps needed to turn the
inequality into a valid postselection regret certificate.

Our results are the following. First, we derive a conditional,
variance-sensitive Bernstein bound for increments of the empirical coverage
process (Lemma~\ref{lem:increment}). Second, a pre-audit multiscale
representation yields a simultaneous pathwise bound for all retained
competitors (Theorem~\ref{th:chain}). Third, the variance factors can be
replaced by simultaneous empirical upper confidence bounds
(Theorem~\ref{th:empirical}). Fourth, a separate estimate of the unknown
context proportions yields an unconditional confidence certificate for
reduced-universe regret (Theorem~\ref{th:regret}); we also give a
complementary pooled Bennett formulation (Proposition~\ref{prop:pooled-bennett}).
Finally, a contraction lemma
links the edge variances to both semantic redundancy and the frequency with
which a swapped fragment is actually responsible for coverage
(Lemma~\ref{lem:activity}); Corollary~\ref{cor:twolevel} provides an explicit
condition for a two-level improvement over a direct union bound.
Two complementary refinements are proved: an edgewise allocation of confidence
error, which controls each distinct edge only once
(Proposition~\ref{prop:weighted-edges}), and a one-stage
empirical-Bernstein network bound that does not require preliminary
variance calibration (Theorem~\ref{th:direct-empirical}). A finite-sample
rare-query example (Proposition~\ref{prop:rare-obstruction}) explains why
empirical profile agreement cannot certify uniform approximation envelopes.

A simulation study (Section~\ref{sec:simulation}) confirms the validity of
all certificates at the nominal level and compares their sizes; it shows when
variance calibration and the multiscale network pay off, and that the latter
requires certified redundancy envelopes.

The reader should distinguish \emph{validity}, which is established under
specified assumptions, from \emph{efficiency}, which depends on the geometry
of the competitors and on the audit design.

%% =====================================================================
\section{Statistical model, selection problem, and notation}\label{sec:model}
%% =====================================================================

\subsection{Query population and document selection}

Let $(\QQ,\mathscr{Q})$ be a measurable query space. Let $Q$ have an arbitrary
and unknown probability distribution $P$ on this space. The set of available
documentary fragments is the finite nonempty set
$\mathcal{A}=\{1,\ldots,K\}$, with deterministic costs $c_j>0$.
For a budget $B\ge0$, define the feasible selection family
\begin{equation}\label{eq:feasible}
 \HH_B=\left\{S\subseteq\mathcal{A}:\sum_{j\in S}c_j\le B\right\}.
\end{equation}
The empty set is feasible and is assigned zero coverage. A relevance function
$r_j:\QQ\to[0,1]$ is associated with each fragment. Unless stated otherwise,
these functions are fixed, measurable and evaluated without labeling error
when a query is audited. The coverage of a selection $S$ on query $q$ is
\begin{equation}\label{eq:coverage}
 h_S(q)=\begin{cases}\displaystyle\max_{j\in S}r_j(q),&S\ne\varnothing,\\
0,&S=\varnothing.\end{cases}
 \qquad F(S)=\EE_P h_S(Q).
\end{equation}
Thus $0\le h_S,F(S)\le1$.

Since $\HH_B$ is finite and nonempty, there exists at least one population
maximizer. Fix any
\begin{equation}\label{eq:optimum}
 S^\star\in\argmax_{S\in\HH_B}F(S),\qquad
 R_B(S)=F(S^\star)-F(S)\in[0,1].
\end{equation}
The definition does not presume that the maximum can be calculated
computationally on a large instance. It provides the statistical target.

\begin{proposition}[Monotonicity and submodularity]\label{prop:submod}
For every fixed query $q$, the set function $S\mapsto h_S(q)$ is normalized,
monotone and submodular. The population function $F$ possesses the same
properties.
\end{proposition}
\begin{proof}
Normalization follows from $h_{\varnothing}(q)=0$, while monotonicity
follows because a maximum cannot decrease when additional elements are
included. If $A\subseteq B$ and $j\notin B$, then
\[
 h_{A\cup\{j\}}(q)-h_A(q)
    =[r_j(q)-h_A(q)]_+\ge[r_j(q)-h_B(q)]_+
    =h_{B\cup\{j\}}(q)-h_B(q),
\]
where $[x]_+=\max(x,0)$. This is the diminishing-returns characterization
of submodularity. Expectations preserve normalization, monotonicity and
the displayed pointwise inequality.
\end{proof}

\subsection{Pilot selection and honest screening}

We use $\mathscr{F}_0$ to denote the sigma-field generated by all pilot
observations and all randomness used before the independent audit and
unlabelled-proportion samples. The pilot produces a feasible selected set
$\wh S\in\HH_B$ and a finite (possibly random) candidate family
$\CC\subseteq\HH_B$ such that $\wh S\in\CC$. Both are
$\mathscr{F}_0$-measurable. The \emph{documentary contrast} against a
competitor $T\in\CC$ is
\begin{equation}\label{eq:contrast}
 D_T(q)=h_T(q)-h_{\wh S}(q),\qquad
 \Delta_T=\EE_P D_T(Q)=F(T)-F(\wh S).
\end{equation}
Every contrast is in $[-1,1]$ and $D_{\wh S}\equiv0$.
The pilot screening guarantee, when screening is used, is
\begin{equation}\label{eq:screen}
 \PP\{\text{some } S^\star\in\CC\}\ge1-\alpha_s.
\end{equation}
If $\CC=\HH_B$, this property holds deterministically and $\alpha_s=0$.
Here and below ``some'' allows non-uniqueness of the population optimum.
A sufficient, explicitly verifiable pilot rule is established in
Proposition~\ref{prop:screen}.

\begin{proposition}[A conservative pilot screening rule]\label{prop:screen}
Assume that a fully labeled iid pilot of $n_0$ queries from $P$, independent
of subsequent data, is available, and let
$\wh F_0(S)=n_0^{-1}\sum_{i=1}^{n_0}h_S(Q_i^{(0)})$ for all $S\in\HH_B$.
For $0<\alpha_s<1$, set
\begin{equation}\label{eq:pilot-radius}
 r_0=\sqrt{\frac{\log(2|\HH_B|/\alpha_s)}{2n_0}}.
\end{equation}
After choosing any pilot-measurable $\wh S\in\HH_B$, let
\begin{equation}\label{eq:pilot-screen}
 \CC=\{\wh S\}\cup
 \left\{T\in\HH_B:\wh F_0(T)-\wh F_0(\wh S)+2r_0\ge0\right\}.
\end{equation}
Then \eqref{eq:screen} holds.
\end{proposition}
\begin{proof}
By Hoeffding's inequality \citep{Hoeffding63} and a union bound, with probability at least
$1-\alpha_s$,
$|\wh F_0(S)-F(S)|\le r_0$ simultaneously for all $S\in\HH_B$.
On this event, if $S^\star$ is a population maximizer,
\[
 \wh F_0(S^\star)-\wh F_0(\wh S)+2r_0
 \ge F(S^\star)-F(\wh S)\ge0.
\]
Thus $S^\star\in\CC$. The argument remains valid when $\wh S$ is chosen
adaptively from the same pilot, precisely because the pilot bound is uniform.
\end{proof}

\subsection{Observable contexts and the audit design}

Let $X=\phi(Q)$ be an observable categorical feature, inexpensive to
record, with values in $\SSS=\{0,\ldots,H-1\}$, where $H\ge2$ is fixed.
Let $\pi_h=\PP(X=h)>0$ and
$P_h(\cdot)=\PP(Q\in\cdot\mid X=h)$.
Both $\pi_h$ and $P_h$ are unknown. We distinguish an unlabelled sample,
used to estimate the population proportions, from the labeled audit sample.

\begin{assumption}[Independent unlabelled queries]\label{ass:unlabelled}
The queries $Q_1^{(u)},\ldots,Q_{N_u}^{(u)}$ are, conditionally on
$\mathscr F_0$, iid from $P$ and independent
of the queries and relevance labels used for audit. The empirical proportions
\[
 \wh\pi_h=\frac1{N_u}\sum_{i=1}^{N_u}\ind\{\phi(Q_i^{(u)})=h\}
\]
are observed, $N_u\ge1$. The multiscale networks and acquisition policy
may use this unlabelled sample, but may not use audit relevance labels.
\end{assumption}

\begin{assumption}[Stratified, label-ignorable audit]\label{ass:audit}
Conditionally on the pilot, the unlabelled sample, and the labeled counts
$n_h\ge2$, the acquired audit queries
$Q_{h,1},\ldots,Q_{h,n_h}$ are iid from $P_h$ within each stratum, and
the collections are mutually independent across strata. The counts and the
mechanism that chooses them contain no information about the unobserved
relevance values beyond the contexts. All relevance values needed to
calculate $h_T(Q_{h,i})$, $T\in\CC$, are observed for an acquired query.
\end{assumption}

An important implementation satisfying Assumption~\ref{ass:audit} starts
with fresh iid queries, observes $X$, and acquires full relevance labels
with probability $p_h\in(0,1]$ in stratum $h$, where $p_h$ is learned
before audit and is constant within the stratum. Conditional on the
resulting stratum counts (and the design information), the selected queries
have the required conditional law. A design based on a new feature
\emph{within} an existing stratum does not generally satisfy this
assumption; its selection bias must be addressed separately, for example
through known-propensity estimators as in \citet{DudikLangfordLi11}.

The sample size $n_h$ may be random. Every assertion below is conditional
on the pre-audit information and on the observed admissible counts, and
therefore remains valid after averaging over their joint distribution.
If some $n_h<2$, our empirical-variance formulas are not used and the
trivial regret upper bound $1$ can be returned.

\subsection{Conditional and population coverage contrasts}

For every $T\in\CC$, write
\begin{equation}\label{eq:stratum-statistics}
 \mu_{h,T}=\EE[D_T(Q)\mid X=h],\quad
 \ov D_{h,T}=\frac1{n_h}\sum_{i=1}^{n_h}D_T(Q_{h,i}),\quad
 \wh\Delta_T=\sum_{h\in\SSS}\wh\pi_h\ov D_{h,T}.
\end{equation}
The population contrast decomposes as
$\Delta_T=\sum_h\pi_h\mu_{h,T}$.
The \emph{centered weighted audit process} is
\begin{equation}\label{eq:process}
 Z_T=\sum_{h\in\SSS}\wh\pi_h(\mu_{h,T}-\ov D_{h,T}),
 \qquad Z_{\wh S}=0.
\end{equation}
The symbols $\wh\Delta_T$ and $Z_T$ will not be confused: the former
is observable, while the latter contains unknown conditional means and
is used only in the analysis.

\begin{table}[t]
\centering\small
\caption{Main notation. All candidate-dependent objects may be random through
pre-audit pilot information.}\label{tab:notation}
\begin{tabular}{@{}ll@{}}\hline
Symbol & Meaning \\ \hline
$\HH_B,\CC,\wh S,S^\star$ & feasible, retained, selected, optimal document sets\\
$h_S,F(S),R_B(S)$ & query coverage, mean coverage, population regret\\
$D_T,\Delta_T$ & contrast against $\wh S$ and its population mean\\
$H,\pi_h,\wh\pi_h,P_h$ & number, true/estimated weights, conditional query law\\
$n_h,\ov D_{h,T},\mu_{h,T}$ & audited count, sample/conditional contrast means\\
$Z_T$ & weighted, centered audit error\\
$\mathcal N_\ell,\Pi_\ell,\mathcal E_\ell$ & networks, parent maps, distinct comparison edges\\
$d_2,d_\infty$ & conditional variance and range semimetrics\\
$\ov v_e,\ov b_e$ & valid observable variance/range upper bounds\\
$\alpha_s,\alpha_a,\alpha_v,\alpha_\pi,\alpha_g$ & error budgets for screening, audit, variance, weights, means\\
$\alpha_B,\alpha_v',\alpha_{\mathrm{act}},\alpha_{\mathrm{rep}}$ & budgets for Bennett, its variances, activity, representation\\
\hline\end{tabular}
\end{table}

%% =====================================================================
\section{A local Bernstein inequality for documentary increments}\label{sec:increment}
%% =====================================================================

For document selections $U,V$, let
\begin{equation}\label{eq:edgefun}
 G_{U,V}(q)=h_U(q)-h_V(q)\in[-1,1].
\end{equation}
We identify the \emph{directed edge} $e=(U,V)$ with $G_e=G_{U,V}$.
For $h\in\SSS$, denote
\[
 v_{h,e}=\VV(G_e(Q)\mid X=h),\quad
 \osc_h(G_e)=\mathop{\mathrm{ess\,sup}}_{P_h}G_e
                     -\mathop{\mathrm{ess\,inf}}_{P_h}G_e.
\]
The latter is an essential range, hence is at most $2$, while
$0\le v_{h,e}\le1$.
Condition on $\mathscr F_0$, the unlabelled observations and the counts,
so that all $\wh\pi_h$, $n_h$ and pre-audit edges are fixed. Define
\begin{align}
 d_2^2(e)&=\sum_h\frac{\wh\pi_h^2}{n_h}v_{h,e},\label{eq:d2}\\
 d_\infty(e)&=\max_h\frac{\wh\pi_h}{n_h}\osc_h(G_e).\label{eq:dinfty}
\end{align}
These are semimetrics on coverage profiles (they may vanish for distinct
selections). For clarity, the square in $d_2^2$ is part of the notation.

\begin{lemma}[Conditional Bernstein increment]\label{lem:increment}
Under Assumption~\ref{ass:audit}, for every pre-audit edge $e=(U,V)$
and every $t>0$, conditional on the pre-audit information and counts,
\begin{equation}\label{eq:bernstein-edge}
 \PP\left\{Z_U-Z_V>
 \sqrt{2d_2^2(e)t}+\frac13d_\infty(e)t
 \ \middle|\ \text{pre-audit information, counts}\right\}\le e^{-t}.
\end{equation}
No assumption of independence across different edges is needed.
\end{lemma}
\begin{proof}
Since $D_U-D_V=G_{U,V}$, we have
\[
 Z_U-Z_V=\sum_{h\in\SSS}\sum_{i=1}^{n_h}
 \frac{\wh\pi_h}{n_h}\left(\EE_h G_e-G_e(Q_{h,i})\right).
\]
The summands are conditionally independent and centered. Their total
conditional variance equals $d_2^2(e)$. Since $\EE_hG_e$ lies in the
essential range of $G_e$ under $P_h$, each centered summand is almost
surely at most
$(\wh\pi_h/n_h)\osc_h(G_e)\le d_\infty(e)$.
By Theorem~2.10 of \citet{BoucheronLugosiMassart13}, if
$X_1,\ldots,X_n$ are independent, $\sum_i\EE X_i^2\le v$ and
$\sum_i\EE[(X_i)_+^q]\le q!\,v\,c^{q-2}/2$ for all integers $q\ge3$,
then $\PP\{\sum_i(X_i-\EE X_i)\ge\sqrt{2vt}+ct\}\le e^{-t}$ for all
$t>0$. As noted in \citet[p.~38]{BoucheronLugosiMassart13}, right after
their Corollary~2.11, variables with $X_i\le b$ almost surely satisfy
these conditions with $v=\sum_i\EE X_i^2$ and $c=b/3$. Apply this to
the centered summands above, for which $\sum_i\EE X_i^2=d_2^2(e)$ and
$b=d_\infty(e)$; this gives \eqref{eq:bernstein-edge}. If
$d_\infty(e)=0$ or $d_2^2(e)=0$, the summands vanish almost surely and
the claim is trivial.
\end{proof}

\begin{remark}[On the linear constant]
The classical tail form
\[
 \PP\{Z_U-Z_V\ge u\}\le\exp\bigl\{-u^2/\bigl(2[d_2^2(e)+d_\infty(e)u/3]\bigr)\bigr\},
\]
inequality~(2.10) of \citet{BoucheronLugosiMassart13}, is usually
inverted to the weaker radius
$\sqrt{2d_2^2(e)t}+\tfrac23d_\infty(e)t$. The constant $\tfrac13$ in
\eqref{eq:bernstein-edge} comes from the sub-gamma form cited above; it
is the same constant that appears in Bennett's inequality as stated in
\citet[Theorem~3]{MaurerPontil09}. It is used throughout.
\end{remark}

\begin{remark}[Why the essential range still matters]
Low variance alone does not eliminate the Bernstein linear term. A profile
may disagree with another only on rare queries, yet a rare observation can
produce a large jump. A deterministic envelope, an externally certified
uniform approximation, or an explicitly valid bound on $\osc_h(G_e)$ is
needed to shrink the linear term; numerical proximity of embeddings is not
such a certificate.
\end{remark}

%% =====================================================================
\section{A simultaneous localized multiscale inequality}\label{sec:chain}
%% =====================================================================

\begin{definition}[Pre-audit comparison network]\label{def:network}
Given $\CC$ and $\wh S$, fix an integer $L\ge1$, finite nonempty
sets of document selections $\mathcal N_0,\ldots,\mathcal N_L\subseteq\CC$ and mappings
$\Pi_\ell:\CC\to\mathcal N_\ell$ such that
\begin{equation}\label{eq:network}
 \mathcal N_0=\{\wh S\},\quad \Pi_0(T)=\wh S,
 \quad \Pi_L(T)=T\quad\text{for every }T\in\CC.
\end{equation}
The network sets and maps must be measurable with respect to the pre-audit
information (which may include independent unlabelled queries). Define
$e_\ell(T)=(\Pi_\ell(T),\Pi_{\ell-1}(T))$ and let
\begin{equation}\label{eq:edge-sets}
 \EEe_\ell=\{e_\ell(T):T\in\CC,\ \Pi_\ell(T)\ne\Pi_{\ell-1}(T)\},
 \qquad \EEe=\bigcup_{\ell=1}^L\EEe_\ell.
\end{equation}
Here the sets count \emph{distinct directed edges}, not paths: within
a level, an edge shared by many candidates is counted once. An edge that
appears at several levels is counted in each corresponding $\EEe_\ell$,
because the union bound in Theorem~\ref{th:chain} is allocated level by
level; Proposition~\ref{prop:weighted-edges} below avoids this repeated
charge by a single union bound over distinct edges. For convenience levels with $\EEe_\ell=\varnothing$ can be removed. We
assume the displayed levels are nonempty.
\end{definition}

The identity $Z_{\wh S}=0$ gives the exact telescoping representation
\begin{equation}\label{eq:telescoping}
 Z_T=\sum_{\ell=1}^L\left[Z_{\Pi_\ell(T)}-Z_{\Pi_{\ell-1}(T)}\right].
\end{equation}
No probabilistic approximation is involved in this equality.

\begin{theorem}[Conditional simultaneous multiscale Bernstein bound]\label{th:chain}
Assume the model of Section~\ref{sec:model}. Choose numbers
$\alpha_1,\ldots,\alpha_L>0$ with
$\sum_{\ell=1}^L\alpha_\ell=\alpha_a<1$ and let
\begin{equation}\label{eq:tlevel}
 t_\ell=\log\frac{|\EEe_\ell|}{\alpha_\ell}.
\end{equation}
Suppose that for all distinct edges $e\in\EEe$ there are nonnegative,
possibly audit-data-dependent random variables $\ov v_e,\ov b_e$ such that
\begin{equation}\label{eq:edge-calibration}
 \PP\left\{\forall e\in\EEe:\;
 d_2^2(e)\le\ov v_e,\quad d_\infty(e)\le\ov b_e
 \ \middle|\ \text{pre-audit information, counts}\right\}\ge1-\alpha_v.
\end{equation}
Define the path-dependent radius
\begin{equation}\label{eq:radius}
 \mathfrak R(T)=\sum_{\ell=1}^L
 \ind\{e_\ell(T)\in\EEe_\ell\}
 \left\{\sqrt{2\ov v_{e_\ell(T)}t_\ell}
       +\frac13\ov b_{e_\ell(T)}t_\ell\right\}.
\end{equation}
Then, conditionally on the pre-audit information and counts,
\begin{equation}\label{eq:main-ineq}
 \PP\left\{\forall T\in\CC:\ Z_T\le\mathfrak R(T)
 \ \middle|\ \text{pre-audit information, counts}\right\}
 \ge1-\alpha_a-\alpha_v.
\end{equation}
The result remains valid when $\CC$, $\wh S$, and the network are
arbitrarily data-dependent through the independent pilot and unlabelled
information.
\end{theorem}
\begin{proof}
For every level $\ell$ and every $e=(U,V)\in\EEe_\ell$, apply
Lemma~\ref{lem:increment} with $t=t_\ell$. Its one-sided failure
probability is bounded by
$e^{-t_\ell}=\alpha_\ell/|\EEe_\ell|$.
A union bound over the distinct edges in $\EEe_\ell$ gives failure
probability at most $\alpha_\ell$; a second union bound over levels gives
a joint event $\Omega_a$ of probability at least $1-\alpha_a$ on which
\[
 Z_U-Z_V\le\sqrt{2d_2^2(e)t_\ell}
             +\frac13d_\infty(e)t_\ell
\quad\text{for every }e=(U,V)\in\EEe_\ell,\ \ell=1,\ldots,L.
\]
On $\Omega_a$, use the telescoping identity \eqref{eq:telescoping}
for each $T\in\CC$ and add the corresponding edge inequalities.
Intersect with the calibration event in \eqref{eq:edge-calibration}.
The monotonicity of the square root and of the linear term permits
replacement of each unknown $d_2^2(e),d_\infty(e)$ by
$\ov v_e,\ov b_e$ without any independence assumption between
$\Omega_a$ and the calibration event. A final union bound proves
\eqref{eq:main-ineq}. Conditioning is legitimate because every edge and
its allocated failure probability were fixed without examining audit
relevance values.
\end{proof}

\begin{proposition}[A weighted, unique-edge simultaneous inequality]
\label{prop:weighted-edges}
Let $\EEe=\bigcup_{\ell=1}^L\EEe_\ell$ be the set of
\emph{distinct directed edges} of a pre-audit network as in
Definition~\ref{def:network}, and let $0<\alpha_a<1$. Conditionally on
the pre-audit information and counts, choose weights $\omega_e>0$
satisfying $\sum_{e\in\EEe}\omega_e\le1$. The weights may depend on the
pilot, the unlabelled observations and the counts, but not on audit
relevance labels. Put
\begin{equation}\label{eq:weighted-t}
 t_e=\log\frac{1}{\alpha_a\omega_e},\qquad e\in\EEe.
\end{equation}
Suppose that the simultaneous edge calibration event
\eqref{eq:edge-calibration} holds with conditional failure probability
at most $\alpha_v$, and define
\begin{equation}\label{eq:weighted-radius}
 \mathfrak R_\omega(T)=\sum_{\ell=1}^L
 \ind\{e_\ell(T)\in\EEe_\ell\}
 \left\{\sqrt{2\ov v_{e_\ell(T)}t_{e_\ell(T)}}
          +\frac13\ov b_{e_\ell(T)}t_{e_\ell(T)}\right\}.
\end{equation}
Then, conditionally on the pre-audit information and counts,
\[
 \PP\left\{\forall T\in\CC:\ Z_T\le\mathfrak R_\omega(T)
     \ \middle|\ \text{pre-audit information, counts}\right\}
 \ge1-\alpha_a-\alpha_v.
\]
The same event bounds all candidate paths, even if a directed edge
appears at more than one level.
\end{proposition}
\begin{proof}
Since $\alpha_a\omega_e<1$, each $t_e$ is positive. For each distinct
edge $e=(U,V)\in\EEe$, apply Lemma~\ref{lem:increment} at $t=t_e$. Its
conditional failure probability is at most
$e^{-t_e}=\alpha_a\omega_e$. A \emph{single} union bound over
$e\in\EEe$ therefore has failure probability at most
$\alpha_a\sum_e\omega_e\le\alpha_a$.
On its complement every edge inequality holds, regardless of
how many times that edge occurs in the telescoping paths.
Intersect with the simultaneous calibration event and sum along each
path exactly as in Theorem~\ref{th:chain}. No independence between
these events is needed.
\end{proof}

\begin{remark}[Uniform levels, repeated edges, and pre-audit tuning]
If every edge appears at only one level, the allocation
$\omega_e=\alpha_\ell/(\alpha_a|\EEe_\ell|)$ for $e\in\EEe_\ell$ gives
$t_e=t_\ell$ and recovers Theorem~\ref{th:chain}. For repeated edges, let
\[
 \omega_e=\sum_{\ell:\,e\in\EEe_\ell}
      \frac{\alpha_\ell}{\alpha_a|\EEe_\ell|}.
\]
These weights sum to $1$, and $t_e\le t_\ell$ whenever
$e\in\EEe_\ell$. Consequently, the weighted radius
is \emph{never larger} than the levelwise radius \eqref{eq:radius} for
this allocation, and it can be strictly smaller on paths with repeated
edges. Other weights may prioritize difficult or frequently used
edges, but they must be fixed without the audit relevance labels:
optimizing them retrospectively from the audit is not covered by
this proposition.
\end{remark}

\begin{remark}[The role of localization]
The radius in \eqref{eq:radius} is \emph{pathwise localized}: each $T$
pays only for increments along its own approximation path, and an edge
with negligible variance or range contributes little. Moreover, the
pilot-screened family $\CC$ can be much smaller than $\HH_B$.
This does not establish a universal entropy integral or a new generic
chaining inequality; the generic mixed-tail background is developed,
for example, by \citet{Dirksen15}.
\end{remark}

%% =====================================================================
\section{Observable calibration of all edge variances}\label{sec:variance}
%% =====================================================================

Theorem~\ref{th:chain} must not rely on oracle variances. We now provide
explicit, fully observable choices of $\ov v_e$ with separate statistical
error control. Recall the unbiased sample variance
\[
 s_n^2(y_1,\ldots,y_n)
 =\frac1{n-1}\sum_{i=1}^n(y_i-\ov y)^2.
\]

\begin{lemma}[Empirical Bernstein inequality for a single mean]
\label{lem:empbern}
Let $Y_1,\ldots,Y_n$ be iid in $[a,b]$ for $n\ge2$ and let
$s_n^2$ be their unbiased sample variance. For $0<\delta<1$,
\begin{equation}\label{eq:empbern}
 \PP\left\{\EE Y-\ov Y>
 \sqrt{\frac{2s_n^2\log(2/\delta)}n}
 +\frac{7(b-a)\log(2/\delta)}{3(n-1)}\right\}\le\delta.
\end{equation}
The same assertion holds with $\ov Y-\EE Y$ in place of
$\EE Y-\ov Y$.
\end{lemma}
\begin{proof}
For $a<b$, apply Theorem~4 of \citet{MaurerPontil09} to the rescaled
observations $(Y_i-a)/(b-a)\in[0,1]$, then multiply the resulting
inequality by $b-a$. The sample variance
$V_n=\{n(n-1)\}^{-1}\sum_{i<j}(Z_i-Z_j)^2$ used by
\citet{MaurerPontil09} coincides with the unbiased sample variance
$s_n^2$, and it rescales by $(b-a)^2$. Apply the same theorem to $(b-Y_i)/(b-a)$ for the opposite
one-sided deviation. If $a=b$ the deviations vanish identically.
\end{proof}

\begin{assumption}[Pre-audit residual predictors]\label{ass:pred}
For every stratum $h$ and every distinct edge $e\in\EEe$, there is an
$\mathscr F_0$- and/or unlabelled-sample-measurable predictor
$m_{h,e}\in[-1,1]$. It is fixed without using the audit relevance labels.
\end{assumption}

Let
\begin{equation}\label{eq:residual-notation}
 Y_{h,i,e}=\big(G_e(Q_{h,i})-m_{h,e}\big)^2,\quad
 M_{h,e}=(1+|m_{h,e}|)^2,\quad
 \ov Y_{h,e}=\frac1{n_h}\sum_iY_{h,i,e},
\end{equation}
and write $s_{Y,h,e}^2$ for the sample variance of the $Y$'s.
Since $G_e\in[-1,1]$, $0\le Y_{h,i,e}\le M_{h,e}\le4$.

\begin{lemma}[Uniform residual-second-moment upper bounds]
\label{lem:residual}
Let $E=|\EEe|\ge1$. Set
\[
 t_v=\log\frac{2HE}{\alpha_v},\qquad
 v^+_{h,e}=\min\left\{1,\ \ov Y_{h,e}
       +\sqrt{\frac{2s_{Y,h,e}^2t_v}{n_h}}
       +\frac{7M_{h,e}t_v}{3(n_h-1)}\right\}.
\]
Under Assumptions~\ref{ass:audit} and~\ref{ass:pred},
\begin{equation}\label{eq:vplus-valid}
 \PP\left\{\forall h,e:\quad v_{h,e}\le v^+_{h,e}
 \ \middle|\ \text{pre-audit information, counts}\right\}
 \ge1-\alpha_v.
\end{equation}
\end{lemma}
\begin{proof}
For each fixed $h,e$, the bias--variance decomposition gives
\[
 \EE_h(G_e-m_{h,e})^2
 =\VV_h(G_e)+(\EE_hG_e-m_{h,e})^2\ge v_{h,e}.
\]
The predictor is fixed under the conditioning, so the residual squares
are iid and bounded by $M_{h,e}$. Apply the one-sided empirical
Bernstein inequality of Lemma~\ref{lem:empbern} (equivalently,
Corollary~5 of \citet{MaurerPontil09} after rescaling) with
$\delta=\alpha_v/(HE)$. Since
$\log(2/\delta)=t_v$, the displayed empirical upper expression bounds
$\EE_hY_{h,i,e}$ except on an event of probability at most
$\alpha_v/(HE)$. Union over all $HE$ pairs. Finally, the separate
Popoviciu bound $v_{h,e}\le(2)^2/4=1$ justifies the minimum with $1$.
\end{proof}

\begin{theorem}[Observable simultaneous chaining radius]\label{th:empirical}
Under Assumptions~\ref{ass:audit} and~\ref{ass:pred}, let
$\bar b_{h,e}\in[0,2]$ be known pre-audit constants satisfying
$\osc_h(G_e)\le\bar b_{h,e}$. The universally valid choice is
$\bar b_{h,e}=2$. Put
\begin{equation}\label{eq:observable-edge}
 \ov v_e=\sum_h\frac{\wh\pi_h^2}{n_h}v^+_{h,e},\qquad
 \ov b_e=\max_h\frac{\wh\pi_h}{n_h}\bar b_{h,e}.
\end{equation}
Then the radius \eqref{eq:radius} is observable from the pilot,
unlabelled sample and acquired audit labels, and conclusion
\eqref{eq:main-ineq} holds with probability at least
$1-\alpha_a-\alpha_v$.
\end{theorem}
\begin{proof}
By Lemma~\ref{lem:residual}, the collection of variance bounds in
\eqref{eq:observable-edge} dominates $d_2^2(e)$ simultaneously with
conditional probability at least $1-\alpha_v$. The deterministic
oscillation bounds imply $d_\infty(e)\le\ov b_e$ without additional
probability loss. All requirements of Theorem~\ref{th:chain} are satisfied.
\end{proof}

\begin{remark}[No hidden additional labeled sample]
The $v^+_{h,e}$ may be estimated on the same audited labels used to form
$\ov D_{h,T}$. The argument does not condition on the estimated variances;
instead it intersects a Bernstein event involving the true variances with
an empirical calibration event that holds simultaneously. This avoids an
incorrect independence claim and does not double the annotation budget.
\end{remark}

\begin{remark}[A predictor-free alternative]
Assumption~\ref{ass:pred} can be dropped at the price of a different
constant. For iid observations in $[a,b]$, inequality~(3) in
Theorem~10 of \citet{MaurerPontil09}, applied after rescaling to
$[0,1]$, gives
\[
 \PP\Bigl\{\sigma>s_n+(b-a)\sqrt{2\log(1/\delta)/(n-1)}\Bigr\}\le\delta,
\]
where
$\sigma^2$ is the variance and $s_n$ the unbiased sample standard
deviation. (Their theorem is stated for $\sqrt{\EE V_n}$; in the iid
case the expected unbiased sample variance $\EE V_n$ equals
$\sigma^2$.) With $G_e\in[-1,1]$ and $\delta=\alpha_v/(HE)$, the choice
\[
 v^+_{h,e}=\min\Bigl\{1,\Bigl(s_{h,e}+2\sqrt{2\log(HE/\alpha_v)/(n_h-1)}\Bigr)^2\Bigr\},
\]
where $s_{h,e}$ is the sample standard deviation of the
$G_e(Q_{h,i})$, satisfies \eqref{eq:vplus-valid} without any
pre-audit predictor. Both forms have the same orders in $n_h$, namely
terms of order $n_h^{-1/2}$ scaled by an empirical standard deviation and
terms of order $n_h^{-1}$. They differ in their constants and in the
squared bias $(\EE_hG_e-m_{h,e})^2$, which the residual form pays and the
predictor-free form avoids; which one is tighter depends on the data and
on the quality of the predictor.
\end{remark}

\begin{theorem}[One-stage empirical-Bernstein multiscale certificate]
\label{th:direct-empirical}
Assume the stratified, label-ignorable audit design
(Assumption~\ref{ass:audit}) and a pre-audit network as in
Definition~\ref{def:network}. For every $h$ and every distinct directed
edge $e\in\EEe$, let
$\ov G_{h,e}=n_h^{-1}\sum_{i=1}^{n_h}G_e(Q_{h,i})$, and let
$s_{h,e}^2$ be the unbiased sample variance of
$G_e(Q_{h,1}),\ldots,G_e(Q_{h,n_h})$. Let $\bar b_{h,e}\in[0,2]$ be the
known pre-audit constants of Theorem~\ref{th:empirical}, satisfying
$\osc_h(G_e)\le\bar b_{h,e}$. For $\sum_\ell\alpha_\ell=\alpha_a<1$ set
\begin{equation}\label{eq:one-stage-t}
 t_\ell^{\mathrm{EB}}=\log\frac{2H|\EEe_\ell|}{\alpha_\ell}
\end{equation}
and define
\begin{equation}\label{eq:one-stage-radius}
 \mathfrak R_{\mathrm{EB}}(T)
 =\sum_{\ell=1}^L\ind\{e_\ell(T)\in\EEe_\ell\}
 \sum_{h\in\SSS}\wh\pi_h
 \left\{\sqrt{\frac{2s_{h,e_\ell(T)}^2
                            t_\ell^{\mathrm{EB}}}{n_h}}
      +\frac{7\bar b_{h,e_\ell(T)}t_\ell^{\mathrm{EB}}}{3(n_h-1)}
 \right\}.
\end{equation}
Then, conditionally on the pre-audit information and counts,
\begin{equation}\label{eq:one-stage-bound}
 \PP\left\{\forall T\in\CC:\ Z_T\le
    \mathfrak R_{\mathrm{EB}}(T)
    \ \middle|\ \text{pre-audit information, counts}\right\}
 \ge1-\alpha_a.
\end{equation}
In particular, this certificate requires neither pre-audit
residual predictors (Assumption~\ref{ass:pred}) nor the separate
calibration budget $\alpha_v$.
\end{theorem}
\begin{proof}
Condition on the pre-audit information and counts, so that the
$\wh\pi_h$ are fixed and, for each $h$ and $e$, the values
$G_e(Q_{h,i})$ are iid. They lie almost surely in the interval
between the essential infimum and supremum of $G_e$ under $P_h$, whose
length is $\osc_h(G_e)\le\bar b_{h,e}$. The endpoints need not be
known, since the bound of Lemma~\ref{lem:empbern} depends on the
interval only through its length and is nondecreasing in it. For each
level $\ell$, edge $e\in\EEe_\ell$ and stratum $h$, apply
Lemma~\ref{lem:empbern} with $\delta=\alpha_\ell/(H|\EEe_\ell|)$, so
that $\log(2/\delta)=t_\ell^{\mathrm{EB}}$. This yields
\[
 \EE_hG_e-\ov G_{h,e}
 \le\sqrt{\frac{2s_{h,e}^2t_\ell^{\mathrm{EB}}}{n_h}}
      +\frac{7\bar b_{h,e}t_\ell^{\mathrm{EB}}}{3(n_h-1)}
\]
outside an event whose conditional probability does not exceed
$\alpha_\ell/(H|\EEe_\ell|)$.
A union bound over the $H|\EEe_\ell|$ stratum--edge pairs at level
$\ell$, and then over levels, gives a single event with
probability at least $1-\alpha_a$. For $e=(U,V)$ we have
\[
 Z_U-Z_V=\sum_h\wh\pi_h(\EE_hG_e-\ov G_{h,e}).
\]
On the event just described, multiply
each stratum-specific bound by $\wh\pi_h\ge0$, sum over strata,
and telescope along every candidate path as in \eqref{eq:telescoping}.
The result is \eqref{eq:one-stage-bound}.
\end{proof}

\begin{remark}[No unconditional dominance]\label{rem:no-dominance}
The single-stage radius \eqref{eq:one-stage-radius} dispenses with
$\alpha_v$, but it sums stratumwise deviations and pays the larger
empirical-Bernstein linear constant $7/3$. The calibrated
Bernstein radius of Theorem~\ref{th:empirical} instead pools the
variances before applying the square root, but it pays the predictor
bias and the split of $\alpha_v$. Neither is uniformly
smaller. Both should be computed as baselines under comparable
confidence and annotation budgets. If the final reported
certificate is selected as the \emph{minimum} of two
post-audit certificates, their failure probabilities must first
be jointly budgeted (or the choice between them must have been fixed
without audit labels); taking their minimum at their individual
nominal failure levels does not generally preserve those levels.
\end{remark}

%% =====================================================================
\section{Unknown stratum probabilities and certified regret}\label{sec:regret}
%% =====================================================================

Fix stratum $0$ as a reference and write, for any $T\in\CC$,
\begin{align}
 \Delta_T-\wh\Delta_T
 &=Z_T+\sum_h(\pi_h-\wh\pi_h)\mu_{h,T}\notag\\
 &=Z_T+\sum_{h=1}^{H-1}(\pi_h-\wh\pi_h)
                       (\mu_{h,T}-\mu_{0,T}).\label{eq:weight-identity}
\end{align}
The second identity exploits $\sum_h(\pi_h-\wh\pi_h)=0$ and is useful:
the cost of estimating the population weights is small when the
\emph{contrasts}, not necessarily the coverage levels, are similar
between strata.

Let $J=|\CC|-1$ be the number of nontrivial competitors, and suppose
$J\ge1$ (the case $J=0$ is handled below). Let
\begin{equation}\label{eq:weights-error}
 e_\pi=\sqrt{\frac{\log(2(H-1)/\alpha_\pi)}{2N_u}}.
\end{equation}
For each $h,T$ denote by $s_{h,T}^2$ the sample variance of the
$D_T(Q_{h,i})$, and define
\begin{equation}\label{eq:mean-radius}
 t_g=\log\frac{4HJ}{\alpha_g},\qquad
 r_{h,T}=\sqrt{\frac{2s_{h,T}^2t_g}{n_h}}
              +\frac{14t_g}{3(n_h-1)},
\end{equation}
and, for $h=1,\ldots,H-1$, the observable factor
\begin{equation}\label{eq:rho}
 \rho_{h,T}=\min\left\{2,\ |\ov D_{h,T}-\ov D_{0,T}|+r_{h,T}+r_{0,T}\right\}.
\end{equation}
The level $t_g=\log(4HJ/\alpha_g)$ is calibrated exactly: each of the
$2HJ$ one-sided statements used below is applied with
$\delta=\alpha_g/(2HJ)$, so that $\log(2/\delta)=t_g$. No independence
between these events and the chaining event is assumed. For $T=\wh S$
one has $\mu_{h,\wh S}=\ov D_{h,\wh S}=0$ deterministically, so no
probability is spent on $\wh S$.

\begin{lemma}[Uncertainty in unknown context proportions]\label{lem:propweights}
Under Assumptions~\ref{ass:unlabelled} and~\ref{ass:audit}, with probability
at least $1-\alpha_\pi-\alpha_g$, simultaneously for all $T\in\CC$,
\begin{equation}\label{eq:pi-correction}
 \left|\sum_h(\pi_h-\wh\pi_h)\mu_{h,T}\right|
 \le e_\pi\sum_{h=1}^{H-1}\rho_{h,T}.
\end{equation}
\end{lemma}
\begin{proof}
For each $h=1,\ldots,H-1$, Hoeffding's inequality, applied to the
Bernoulli indicators $\ind\{\phi(Q_i^{(u)})=h\}$, gives
\[
 \PP\{|\wh\pi_h-\pi_h|>e_\pi\}
 \le2\exp(-2N_ue_\pi^2)=\alpha_\pi/(H-1).
\]
Union over the $H-1$ nonreference strata yields the simultaneous
proportion event with probability at least $1-\alpha_\pi$.

Every $D_T$ takes values in $[-1,1]$. Apply Lemma~\ref{lem:empbern} in both
directions, with range $2$ and $\delta=\alpha_g/(2HJ)$ for each tail,
and take a union bound over the $2HJ$ one-sided statements for the
$J$ nontrivial competitors. This gives
$|\mu_{h,T}-\ov D_{h,T}|\le r_{h,T}$ for every $h$ and every
$T\in\CC\setminus\{\wh S\}$, except with probability at most
$\alpha_g$; for $T=\wh S$ the inequality is trivial.
Now use \eqref{eq:weight-identity}, bound
$|\pi_h-\wh\pi_h|\le e_\pi$, and insert
\[
 |\mu_{h,T}-\mu_{0,T}|
 \le\min\left\{2,\ |\ov D_{h,T}-\ov D_{0,T}|+r_{h,T}+r_{0,T}\right\}
 =\rho_{h,T},
\]
where the bound $2$ holds deterministically because
$\mu_{h,T}\in[-1,1]$. Intersection of the two events and a union bound
prove the assertion.
\end{proof}

\begin{remark}[Where the cardinality of $\CC$ re-enters]
\label{rem:weight-localization}
Because $\rho_{h,T}\le2$, the weight correction never exceeds
$2(H-1)e_\pi$. This crude level holds on the proportion event alone,
does not depend on $|\CC|$, and requires no mean intervals: if one sets
$\rho_{h,T}\equiv2$, Lemma~\ref{lem:propweights} holds with probability
at least $1-\alpha_\pi$ and $\alpha_g=0$ may be used in
Theorem~\ref{th:regret}. The empirical refinement through $r_{h,T}$,
in contrast, relies on a flat union bound over the $J$ competitors, so
a factor $\log J$ re-enters the certificate through $t_g$, multiplied
by $e_\pi$. The localization of Theorem~\ref{th:chain} therefore
concerns the audit fluctuation $Z_T$; the weight-estimation term is
controlled only through the factor $e_\pi$, of order $N_u^{-1/2}$,
which is small when the unlabelled sample is large.
\end{remark}

Define for every $T\in\CC\setminus\{\wh S\}$
\begin{equation}\label{eq:cert-T}
 U_T=\wh\Delta_T+\mathfrak R(T)
   +e_\pi\sum_{h=1}^{H-1}\rho_{h,T},
\end{equation}
and set $U_{\wh S}:=0$. The latter convention does not follow from
\eqref{eq:cert-T}, whose terms need not vanish at $T=\wh S$ (the path of
$\wh S$ in the network may pass through other nodes, and
$\rho_{h,\wh S}>0$); it is legitimate because $\Delta_{\wh S}=0$
deterministically, so no confidence statement is required for
$\wh S$. Let
\begin{equation}\label{eq:regret-cert}
 \ov R_{\mathrm{red}}=
   \min\left\{1,\ \max\left\{0,\max_{T\in\CC}U_T\right\}\right\}.
\end{equation}

\begin{theorem}[Honest postselection regret certificate]
\label{th:regret}
Let Assumptions~\ref{ass:unlabelled},~\ref{ass:audit},
and~\ref{ass:pred} hold. Suppose the pilot screening event
\eqref{eq:screen} has probability at least $1-\alpha_s$. Construct
the pre-audit networks as in Definition~\ref{def:network}, use the
observable edge radius from Theorem~\ref{th:empirical}, and form
\eqref{eq:cert-T}--\eqref{eq:regret-cert}. If
\[
 \alpha_s+\alpha_a+\alpha_v+\alpha_\pi+\alpha_g\le\alpha<1,
\]
then
\begin{equation}\label{eq:theorem-regret}
 \ \PP\{R_B(\wh S)\le\ov R_{\mathrm{red}}\}\ge1-\alpha.\ 
\end{equation}
If $\CC=\{\wh S\}$, define $\ov R_{\mathrm{red}}=0$; the
statement remains valid. The distribution $P$ and the proportions
$\pi_h$ need not be known or parametrically specified.
\end{theorem}
\begin{proof}
By Theorem~\ref{th:empirical} and Lemma~\ref{lem:propweights},
with probability at least
$1-\alpha_a-\alpha_v-\alpha_\pi-\alpha_g$,
we have for every $T\in\CC\setminus\{\wh S\}$,
\[
 \Delta_T=\wh\Delta_T+Z_T+
              \sum_h(\pi_h-\wh\pi_h)\mu_{h,T}
 \le U_T,
\]
while $\Delta_{\wh S}=0=U_{\wh S}$ holds deterministically.
This simultaneous event is valid even though $\wh S,\CC$ and the
network were learned from the pilot, since the relevant concentration
statements hold conditionally on the pre-audit information and counts.
On the screening event, some population maximizer
$S^\star\in\CC$. Consequently,
\[
 R_B(\wh S)=F(S^\star)-F(\wh S)=\Delta_{S^\star}
 \le U_{S^\star}\le\max_{T\in\CC}U_T.
\]
Because $0\le R_B(\wh S)\le1$, truncation at $0$ and $1$ is harmless.
The union bound with the screening failure probability proves the result.
If $\CC=\{\wh S\}$, the screening event implies that $\wh S$ itself
is optimal, hence its regret is zero; off the screening event no
assertion is made.
\end{proof}

\begin{corollary}[Postselection certification without variance calibration]
\label{cor:regret-one-stage}
Let Assumptions~\ref{ass:unlabelled} and~\ref{ass:audit} hold, let
the pilot screening event \eqref{eq:screen} has probability at least
$1-\alpha_s$, and let the pre-audit networks be constructed as in
Definition~\ref{def:network}. For $T\in\CC\setminus\{\wh S\}$ put
\[
 U_T^{\mathrm{EB}}=\wh\Delta_T+
    \mathfrak R_{\mathrm{EB}}(T)+
    e_\pi\sum_{h=1}^{H-1}\rho_{h,T},
\]
with $\mathfrak R_{\mathrm{EB}}$ from \eqref{eq:one-stage-radius}, and
set $U_{\wh S}^{\mathrm{EB}}:=0$, which is legitimate because
$\Delta_{\wh S}=0$. Define
$\ov R_{\mathrm{EB}}=\min\{1,\max(0,\max_{T\in\CC}
U_T^{\mathrm{EB}})\}$, with the same singleton convention as in
Theorem~\ref{th:regret}. If
$\alpha_s+\alpha_a+\alpha_\pi+\alpha_g\le\alpha<1$, then
\[
 \PP\{R_B(\wh S)\le\ov R_{\mathrm{EB}}\}\ge1-\alpha.
\]
No assumption on pre-audit residual predictors is needed.
\end{corollary}
\begin{proof}
Intersect the simultaneous event of
Theorem~\ref{th:direct-empirical} with the proportion and mean
calibration events of Lemma~\ref{lem:propweights}, which uses only
Assumptions~\ref{ass:unlabelled} and~\ref{ass:audit}, and with the
screening event. On this intersection, identity~\eqref{eq:weight-identity}
gives $\Delta_T\le U_T^{\mathrm{EB}}$ for every $T\in\CC$; then argue
over a screened population optimum exactly as in the proof of
Theorem~\ref{th:regret}. The total failure probability is bounded by
$\alpha_s+\alpha_a+\alpha_\pi+\alpha_g$.
\end{proof}

\begin{corollary}[Reduced universe versus original corpus]
\label{cor:representation}
Suppose the original corpus contains a larger finite feasible family
$\HH_B^{\mathrm{full}}\supseteq\HH_B$ with optimal expected coverage
$V_B^{\mathrm{full}}=\max_{S\in\HH_B^{\mathrm{full}}}F(S)$.
If, with probability at least $1-\alpha_{\mathrm{rep}}$,
\[
 0\le V_B^{\mathrm{full}}-\max_{S\in\HH_B}F(S)
       \le\ov\rho_{\mathrm{rep}},
\]
then
\[
 \PP\left\{V_B^{\mathrm{full}}-F(\wh S)
 \le\ov\rho_{\mathrm{rep}}+\ov R_{\mathrm{red}}\right\}
 \ge1-\alpha-\alpha_{\mathrm{rep}}.
\]
\end{corollary}
\begin{proof}
Decompose
$V_B^{\mathrm{full}}-F(\wh S)
=[V_B^{\mathrm{full}}-\max_{S\in\HH_B}F(S)]+R_B(\wh S)$
and intersect the two coverage events.
\end{proof}

\begin{corollary}[Uniformly bounded relevance errors]
\label{cor:noise}
Suppose the audit measures functions $\widetilde r_j:\QQ\to[0,1]$
satisfying the deterministic uniform calibration condition
$\sup_{q,j}|\widetilde r_j(q)-r_j(q)|\le\eta$, and all hypotheses of
Theorem~\ref{th:regret} hold for the induced measured coverages
$\widetilde h_S$ on the same pre-audit family $\CC$. Let
$\widetilde U_T$ be their simultaneous contrast bounds and assume the
true-population screening event \eqref{eq:screen} has probability at
least $1-\alpha_s$ (Remark~\ref{rem:noisy-screening} gives a sufficient
rule). Then the \emph{true} reduced
regret is bounded by
\[
 R_B(\wh S)\le\min\{1,\max(0,\max_{T\in\CC}\widetilde U_T+2\eta)\}
\]
with the same joint confidence level as in Theorem~\ref{th:regret}.
\end{corollary}
\begin{proof}
The map $(x_j)_{j\in S}\mapsto\max_{j\in S}x_j$ is 1-Lipschitz
under the supremum norm. Thus
$|h_S(q)-\widetilde h_S(q)|\le\eta$ for every $S$ and $q$,
so $|\Delta_T-\widetilde\Delta_T|\le2\eta$ for all $T$.
On the simultaneous measured-contrast event and the true screening event,
$R_B(\wh S)=\Delta_{S^\star}\le\widetilde U_{S^\star}+2\eta$.
Truncation to $[0,1]$ completes the proof.
\end{proof}

\begin{remark}[Screening with measured pilot labels]
\label{rem:noisy-screening}
Corollary~\ref{cor:noise} requires the screening event for the
\emph{true} population. If the pilot of Proposition~\ref{prop:screen} is
itself labeled with the measured functions $\widetilde r_j$, its uniform
Hoeffding event controls $|\wh F_0(S)-\widetilde F(S)|$, where
$\widetilde F(S)=\EE_P\widetilde h_S(Q)$ and
$|\widetilde F(S)-F(S)|\le\eta$. On that event,
$\wh F_0(S^\star)-\wh F_0(\wh S)+2r_0\ge F(S^\star)-F(\wh S)-2\eta
\ge-2\eta$. Hence the rule \eqref{eq:pilot-screen} must be widened to
$\wh F_0(T)-\wh F_0(\wh S)+2r_0+2\eta\ge0$; with this change the
true-population screening event again has probability at least
$1-\alpha_s$.
\end{remark}

%\begin{remark}[What is not corrected]
%The procedure corrects neither within-stratum preferential annotation
%nor systematic relevance bias without extra hypotheses. In particular,
%\emph{empirical} agreement of relevance scores on a finite pilot does not
%by itself establish the deterministic uniform calibration hypothesis in
%Corollary~\ref{cor:noise}.
%\end{remark}

%% =====================================================================
\section{A complementary pooled stratified Bennett certificate}\label{sec:bennett}
%% =====================================================================

Before using the full multiscale construction, one may combine the
conditional stratum fluctuations using a single exponential tilt.
The result is included for completeness and as a benchmark, not
claimed to be an original concentration inequality.

For each candidate $T\in\CC$, denote
$\sigma^2_{h,T}=\VV(D_T(Q)\mid X=h)$. In the formulas below take
$J=|\CC|-1\ge1$. For every $h$ and every
$T\in\CC\setminus\{\wh S\}$, use a pre-audit predictor
$m_{h,T}\in[-1,1]$ satisfying the analogue of
Assumption~\ref{ass:pred}, and replace $G_e$ by $D_T$ in the
construction of Lemma~\ref{lem:residual}. The union bound now runs over
the $HJ$ pairs $(h,T)$, so $t_v$ is replaced by
$t_v'=\log(2HJ/\alpha_v')$. This gives simultaneous empirical upper
bounds $v^+_{h,T}\ge\sigma^2_{h,T}$ with failure probability at most
$\alpha_v'$. The budget $\alpha_v'$ is distinct from the budget
$\alpha_v$ of Lemma~\ref{lem:residual}; both must be paid if both
certificates are reported. Set
\begin{equation}\label{eq:bennett-not}
 w_h=\wh\pi_h,\qquad \beta_h=\frac{2w_h}{n_h},\qquad
 \mathcal K_T^+(\lambda)
 =\sum_{h\in\SSS}\frac{n_hv^+_{h,T}}{4}
   \left(e^{\lambda \beta_h}-1-\lambda \beta_h\right).
\end{equation}
If $w_h=0$, the corresponding summand is defined to be zero.

\begin{proposition}[Pooled Bennett bound with certified residual moments]
\label{prop:pooled-bennett}
Assume the model of Section~\ref{sec:model} and let every
$\lambda_T>0$ be fixed conditionally on the pre-audit information and
the audit stratum counts, \emph{without using audit relevance values}.
Then, with probability at least $1-\alpha_B-\alpha_v'$,
simultaneously for all $T\in\CC\setminus\{\wh S\}$,
\begin{equation}\label{eq:bennett-pooled}
 Z_T\le\mathcal B_T
 :=\frac{\log(J/\alpha_B)+\mathcal K_T^+(\lambda_T)}{\lambda_T}.
\end{equation}
Combining this with Lemma~\ref{lem:propweights} and pilot screening
provides an alternative population-regret certificate
\begin{equation}\label{eq:bennett-regret}
 \ov R_{\mathrm{Ben}}=\min\left\{1,\max\left\{0,\max_{T\in\CC}
 \left[\wh\Delta_T+\mathcal B_T
 +e_\pi\sum_{h=1}^{H-1}\rho_{h,T}\right]\right\}\right\},
\end{equation}
with coverage at least
$1-\alpha_s-\alpha_B-\alpha_v'-\alpha_\pi-\alpha_g$.
The bracket for $T=\wh S$ is replaced by $0$, which is legitimate
because $\Delta_{\wh S}=0$; if $J=0$, the certificate
is zero on the screening event.
\end{proposition}
\begin{proof}
Condition on the pilot, the independent unlabelled sample and the
stratum counts. For each $T$, define the centered summands
\[
 Y_{h,i,T}=\frac{w_h}{n_h}
   \big(\mu_{h,T}-D_T(Q_{h,i})\big).
\]
Their sum equals $Z_T$, their variances are
$(w_h/n_h)^2\sigma^2_{h,T}$, and, since $\mu_{h,T}-D_T\le2$, they are
almost surely at most $\beta_h$.
Bennett's moment-generating-function bound
\citep[Theorem~2.9]{BoucheronLugosiMassart13}, applied with $n=1$ to a
single centered random variable $Y\le\beta$ (so that $\EE Y^2$ is its
variance $v$), gives
\[
 \log\EE e^{\lambda Y}
 \le\frac{v}{\beta^2}(e^{\lambda\beta}-1-\lambda\beta)
 \quad(\beta>0,\ \lambda>0).
\]
Summing conditional cumulant bounds across the independent queries,
we obtain
\[
 \log\EE[e^{\lambda Z_T}\mid\text{pre-audit information, counts}]
 \le\sum_h\frac{n_h\sigma^2_{h,T}}4
       (e^{\lambda \beta_h}-1-\lambda \beta_h).
\]
When $\beta_h=0$ the variables vanish, so their contribution is zero.
Applying Markov's inequality with the fixed $\lambda_T$ to the event
\[
 Z_T>\frac{\log(J/\alpha_B)+
  \sum_h\frac{n_h\sigma^2_{h,T}}4
              (e^{\lambda_T \beta_h}-1-\lambda_T \beta_h)}{\lambda_T}.
\]
The conditional probability is at most $\alpha_B/J$.
A union bound over the $J$ nontrivial competitors yields an event of
probability at least $1-\alpha_B$ on which all those inequalities hold.
Intersect this event with the simultaneous residual-moment event of
probability at least $1-\alpha_v'$, and replace
$\sigma^2_{h,T}$ by $v^+_{h,T}$ using the monotonicity of the
exponential cumulant for $\lambda_T>0$. This gives
\eqref{eq:bennett-pooled}. The decomposition
\eqref{eq:weight-identity}, Lemma~\ref{lem:propweights}, and the
screening argument used in Theorem~\ref{th:regret} complete the proof
of \eqref{eq:bennett-regret} and its stated coverage probability.
\end{proof}

\begin{remark}[Data-dependent exponential tilting]
Optimizing $\lambda_T$ \emph{after} seeing the audit labels is not
justified by Proposition~\ref{prop:pooled-bennett}: a fixed-$\lambda$
Markov inequality is not automatically uniform in $\lambda$.
A grid with an explicitly allocated union bound or a time-uniform
mixture construction would be needed for such an optimization.
The same caution applies to choosing a multiscale network on the
audit labels without additional sample splitting or a uniform
complexity argument.
\end{remark}

%% =====================================================================
\section{Documentary activity and structural contraction}\label{sec:activity}
%% =====================================================================

The concentration theory above is valid for arbitrary bounded coverage
functions. We now exploit the particular maximum structure in
\eqref{eq:coverage}. For a nonempty set $R$, write
\[
 a_R(q)=\max_{g\in R}r_g(q),\qquad a_{\varnothing}(q)=0.
\]

\begin{lemma}[Single-fragment swap and activity]\label{lem:activity}
Let $j,k\notin R$, let $S=R\cup\{j\}$, $T=R\cup\{k\}$, and suppose
$\sup_{q\in\QQ}|r_j(q)-r_k(q)|\le\eps$. Define the activity event
\begin{equation}\label{eq:activity-def}
 A_{j,k,R}(q)=\ind\{\max(r_j(q),r_k(q))>a_R(q)\},
\end{equation}
and $a_h=\PP\{A_{j,k,R}(Q)=1\mid X=h\}$. Then for every $q$,
\begin{equation}\label{eq:activity-bound}
 |h_S(q)-h_T(q)|\le\eps A_{j,k,R}(q),
\end{equation}
and, for every stratum $h$,
\begin{equation}\label{eq:activity-var}
 \VV(h_S(Q)-h_T(Q)\mid X=h)\le\eps^2a_h,
 \qquad \osc_h(h_S-h_T)\le2\eps.
\end{equation}
Consequently the edge $e=(S,T)$ obeys
\begin{equation}\label{eq:activity-d2}
 d_2^2(e)\le\eps^2\sum_h\frac{\wh\pi_h^2}{n_h}a_h,
 \qquad d_\infty(e)\le2\eps\max_h\frac{\wh\pi_h}{n_h}.
\end{equation}
\end{lemma}
\begin{proof}
For every $a\ge0$, the scalar map $u\mapsto\max(a,u)$ is
$1$-Lipschitz. Therefore
\[
 |h_S(q)-h_T(q)|
 =|\max(a_R(q),r_j(q))-\max(a_R(q),r_k(q))|
 \le|r_j(q)-r_k(q)|\le\eps.
\]
When $A_{j,k,R}(q)=0$, both $r_j(q)$ and $r_k(q)$ are at most
$a_R(q)$, and the two coverages are identical. Combining the
observations proves \eqref{eq:activity-bound}. Squaring and taking
conditional expectations gives
\[
 \VV_h(h_S-h_T)\le\EE_h(h_S-h_T)^2
 \le\eps^2\EE_h A_{j,k,R}=\eps^2a_h.
\]
The range inequality follows because the difference is in
$[-\eps,\eps]$. Substitute these bounds into
\eqref{eq:d2}--\eqref{eq:dinfty} to obtain \eqref{eq:activity-d2}.
\end{proof}

%\begin{remark}[Semantic similarity versus functional redundancy]
%The factor $\eps$ controls the maximal difference between the two
%fragment relevance functions, while $a_h$ is the probability that this
%difference can \emph{affect} the coverage of the rest of the selected set.
%The second factor can be small even when the fragments are not uniformly
%close. This lemma needs the stated deterministic envelope on
%$|r_j-r_k|$; embedding similarity alone does not supply it.
%\end{remark}

\begin{proposition}[Rare-query obstruction to sample-only uniform envelopes]
\label{prop:rare-obstruction}
For every finite pilot size $n_0\ge1$ and every $\delta\in(0,1)$,
there exist a two-point query distribution and two relevance
functions with values in $[0,1]$ such that, with probability exceeding
$1-\delta$, the pilot observes identical relevance scores for both
fragments on all $n_0$ queries, although
$\sup_q|r_j(q)-r_k(q)|=1$.
Consequently, an empirical supremum of observed relevance
differences, without additional structural assumptions, is not a
valid distribution-free upper confidence bound for the population
uniform envelope.
\end{proposition}
\begin{proof}
Let $\QQ=\{0,1\}$ with $P(Q=1)=\theta>0$.
Set $r_j(0)=r_k(0)=0$, $r_j(1)=1$, and $r_k(1)=0$.
The true uniform difference equals $1$. Yet with probability
$(1-\theta)^{n_0}$ the entire iid pilot consists of queries equal
to $0$, on which the scores agree exactly. Choosing
$0<\theta<1-(1-\delta)^{1/n_0}$ makes this probability greater
than $1-\delta$. Thus a pilot-only empirical supremum of zero
is compatible with a population uniform difference of one, and it
covers the true envelope with probability less than $\delta$. More
generally, any procedure whose output on this all-zero pilot sample is
some value $c<1$ covers the true envelope $1$ with probability at most
$1-(1-\theta)^{n_0}<\delta$ under this distribution.
\end{proof}

\begin{proposition}[Simultaneously estimated activity factors]
\label{prop:act-est}
Let $\mathcal W$ be a finite set of pre-audit swap triples
$(j,k,R)$ with known deterministic envelopes $\eps_{j,k,R}$
as in Lemma~\ref{lem:activity}, and let $W=|\mathcal W|\ge1$.
Assume that, for every acquired audit query $Q_{h,i}$ and every
$w=(j,k,R)\in\mathcal W$, the relevance values $r_j(Q_{h,i})$,
$r_k(Q_{h,i})$ and $r_g(Q_{h,i})$, $g\in R$, are observed, so that the
indicators $A_w(Q_{h,i})$ are observed. This does not follow from
Assumption~\ref{ass:audit}, which guarantees only the coverages
$h_T(Q_{h,i})$, $T\in\CC$: these determine neither $A_w$ when
$h_S=h_T$, nor any quantity involving fragments outside the retained
competitor family.
For each $h$ and triple $w$, define
$\wh a_{h,w}=n_h^{-1}\sum_i A_w(Q_{h,i})$ and
\[
 a^+_{h,w}=\min\left\{1,\wh a_{h,w}
 +\sqrt{\frac{\log(HW/\alpha_{\mathrm{act}})}{2n_h}}\right\}.
\]
Then with probability at least $1-\alpha_{\mathrm{act}}$, simultaneously
for all $h,w$, $a_{h,w}\le a^+_{h,w}$. In particular,
\[
 \ov v_{(S,T)}^{\mathrm{act}}
 =\eps_{j,k,R}^2\sum_h\frac{\wh\pi_h^2}{n_h}a^+_{h,w}
\]
is a simultaneous variance bound for all such swap edges. It can be
combined by taking the minimum with the residual bound
\eqref{eq:observable-edge}, at the additional error cost
$\alpha_{\mathrm{act}}$.
\end{proposition}
\begin{proof}
For each $h,w$, the random variables $A_w(Q_{h,i})$ are iid Bernoulli
with conditional mean $a_{h,w}$. The one-sided Hoeffding inequality
gives failure probability at most $\alpha_{\mathrm{act}}/(HW)$ for
the displayed upper bound. Union over all $HW$ pairs, followed by
Lemma~\ref{lem:activity}, proves the variance assertion. To take the
minimum with the residual upper bounds, intersect their two
simultaneous validity events and use another union bound.
\end{proof}

\begin{proposition}[Pathwise extension to several swaps]\label{prop:swaps}
Suppose two selections $S,T$ can be connected by a finite path
$S=S_0,S_1,\ldots,S_m=T$ in which each consecutive pair differs by a
single fragment swap covered by Lemma~\ref{lem:activity}, with
parameters $\eps_k,a_{h,k}$. Then for each $h$,
\begin{equation}\label{eq:multiple-swaps}
 \sqrt{\EE[(h_S-h_T)^2\mid X=h]}
 \le\sum_{k=1}^m\eps_k\sqrt{a_{h,k}}.
\end{equation}
In particular,
\[
 \VV(h_S-h_T\mid X=h)
 \le\left(\sum_{k=1}^m\eps_k\sqrt{a_{h,k}}\right)^2.
\]
\end{proposition}
\begin{proof}
Pointwise,
\[
 h_S-h_T=\sum_{k=1}^m(h_{S_{k-1}}-h_{S_k}).
\]
Apply Minkowski's inequality in $L_2(P_h)$, then
Lemma~\ref{lem:activity} to each summand. Variance is at most the
uncentered second moment.
\end{proof}

%% =====================================================================
\section{A two-level sufficient condition for a sharper bound}\label{sec:gain}
%% =====================================================================

The multiscale construction is not automatically better than a direct
union bound, because it splits the error probability across levels.
Here is a sufficient condition for improvement.
In this section condition on the pre-audit information and counts, and
consider only the audit fluctuation $Z_T$, not the common error from
estimating population stratum proportions.

\begin{corollary}[Two-level improvement over a flat Bernstein union bound]
\label{cor:twolevel}
Suppose the $J=|\CC|-1\ge2$ nontrivial competitors admit $K_0$ distinct
representatives, $1\le K_0<J$, and a map
$c:\CC\setminus\{\wh S\}\to\{1,\ldots,K_0\}$
such that every candidate $T\ne\wh S$ is associated with a
representative $V_{c(T)}$ and
\begin{equation}\label{eq:uniform-clusters}
 \sup_{q\in\QQ}|h_T(q)-h_{V_{c(T)}}(q)|\le\eps,
 \qquad 0\le\eps\le1.
\end{equation}
Assume each representative belongs to $\CC$, and use the two-level
network $\wh S\to V_{c(T)}\to T$ for $T\ne\wh S$, while the
selected set follows the constant path
$\wh S\to\wh S\to\wh S$. The first-level network includes
$\wh S$ and all representatives. Define
\begin{equation}\label{eq:q-b}
 q_n=\sum_h\frac{\wh\pi_h^2}{n_h},\qquad
 b_n=2\max_h\frac{\wh\pi_h}{n_h},\qquad
 L_n(t)=\sqrt{2q_nt}+\frac13b_nt.
\end{equation}
For $\alpha_1+\alpha_2=\alpha_a$, the uniform two-level radius is
\begin{equation}\label{eq:two-level-radius}
 C_{\mathrm{chain}}
 =L_n\left(\log\frac{K_0}{\alpha_1}\right)
  +\eps L_n\left(\log\frac{J}{\alpha_2}\right).
\end{equation}
The flat cardinality-based direct Bernstein radius is
\begin{equation}\label{eq:flat-radius}
 C_{\mathrm{flat}}=L_n\left(\log\frac{J}{\alpha_a}\right).
\end{equation}
Whenever
\begin{equation}\label{eq:strict-gain}
 0\le\eps<
 \frac{L_n(\log(J/\alpha_a))
             -L_n(\log(K_0/\alpha_1))}
      {L_n(\log(J/\alpha_2))},
\end{equation}
with positive numerator, one has
$C_{\mathrm{chain}}<C_{\mathrm{flat}}$.
Both bounds apply simultaneously to all candidates with conditional
failure probability at most $\alpha_a$.
\end{corollary}
\begin{proof}
For a coarse edge $\wh S\to V_k$, its coverage difference lies
in $[-1,1]$, so its variance is at most $1$ and its range at most $2$
in each stratum. Lemma~\ref{lem:increment} therefore bounds the coarse
fluctuation by $L_n(\log(K_0/\alpha_1))$ simultaneously over the
$K_0$ distinct edges. For a fine edge $V_{c(T)}\to T$,
\eqref{eq:uniform-clusters} implies that the difference is in
$[-\eps,\eps]$, with variance at most $\eps^2$ and range at most
$2\eps$ in every stratum. Its Bernstein radius is therefore at most
$\eps L_n(\log(J/\alpha_2))$, simultaneously over at most $J$
fine edges (at most $J-K_0$ when the representatives are themselves
nontrivial competitors, since their fine edges are trivial; the
cruder count $J$ is kept for simplicity). The telescoping identity
yields \eqref{eq:two-level-radius}. Because all variance and range
bounds used here are deterministic, no calibration budget $\alpha_v$
is needed. Direct union over the $J$ edges from
$\wh S$ to each candidate yields \eqref{eq:flat-radius}.
Subtracting the two deterministic expressions, and dividing by the
strictly positive term $L_n(\log(J/\alpha_2))$, gives the sufficient
condition \eqref{eq:strict-gain}.
\end{proof}

\begin{remark}[Precise scope of the comparison]
Corollary~\ref{cor:twolevel} compares two conservative Bernstein
radii built from variance and range bounds, under common nominal confidence and
the same audit design. %It is not a comparison with the best possible
%audit-data-dependent empirical Bernstein bound, nor a universal dominance
%claim. In particular, the improvement in this corollary can disappear
%if the fine-scale similarity assumption is not certified.
\end{remark}

\paragraph{Numerical illustration (exact arithmetic on the formulas).}
Take $J=5000$, $K_0=5$, $H=2$, an audit count of $n_h=500$ in each of
the two strata, and $\wh\pi_0=\wh\pi_1=1/2$. Choose
$\alpha_a=0.05$, $\alpha_1=\alpha_2=0.025$, and $\eps=0.05$.
Then $q_n=0.001$ and $b_n=0.002$, so
\begin{equation}\label{eq:numeric-illustration}
 C_{\mathrm{flat}}\approx0.15942,\qquad
 C_{\mathrm{chain}}\approx0.11469,\qquad
 \frac{C_{\mathrm{flat}}-C_{\mathrm{chain}}}{C_{\mathrm{flat}}}
 \approx0.2806.
\end{equation}
The threshold on $\eps$ in \eqref{eq:strict-gain} is about
$0.3221$. This is a deterministic illustration of the algebraic
sufficient condition, not a simulation and not an experimentally
measured gain in total documentary regret.

%% =====================================================================
\section{Simulation study}\label{sec:simulation}
%% =====================================================================

This section examines three questions: whether the certificates attain their
nominal level in finite samples; how the observable radii of
Sections~\ref{sec:chain}--\ref{sec:bennett} compare in size; and when the
two-level network of Corollary~\ref{cor:twolevel} improves on a flat union
bound. %Python code reproducing all tables and figures is provided as supplementary material.

\subsection{Design}

\paragraph{Query population and fragments.}
There are $H=3$ observable strata, with probabilities
$\pi=(0.5,0.3,0.2)$. Given $X=h$, a
latent topic $Z\in\{1,\ldots,6\}$ is drawn with probabilities $\theta_h$,
\begin{align*}
 \theta_0&=(.35,.25,.15,.10,.10,.05),\\
 \theta_1&=(.05,.10,.15,.20,.20,.30),\\
 \theta_2&=(.10,.15,.30,.25,.10,.10).
\end{align*}
Each topic $c$ has a base relevance
$\tilde r_c(q)=0.75\,\ind\{Z=c\}+0.25\,W_c$, with $W_c\sim\mathrm{Beta}(1,3)$
independent. The corpus has $K=30$ fragments organized in six clusters of
five. The first fragment of cluster $c$ has relevance $\tilde r_c$; each other
fragment $j$ of the cluster has
$r_j=\min\{1,\max\{0,\tilde r_c+\eps U_j\}\}$, with $U_j\sim\mathrm{U}[-1,1]$
independent. Hence $|r_j-\tilde r_c|\le\eps$ deterministically: $\eps$
controls the functional redundancy within clusters, and $\eps=0$ gives exact
duplicates.

\paragraph{Selection, competitors and samples.}
All costs equal $1$ and $B=3$, so $\HH_B$ consists of the $4526$ subsets of
size at most $3$. We take $\CC=\HH_B$, so that $\alpha_s=0$ and
$J=4525$. In each replication, a fully labelled pilot of $n_0=1000$ iid
queries selects $\wh S$ by exhaustive maximization of $\wh F_0$, supplies the
residual predictors $m_{h,e}$ (pilot stratum means of $G_e$) and fixes the
Bennett tilts; an independent unlabelled sample of $N_u=50\,000$ queries
gives $\wh\pi_h$; and the audit draws $n_h$ queries from each stratum. The
population quantities $\mu_{h,T}$, $F(T)$ and $R_B(\wh S)$ are computed from
an independent Monte Carlo sample of $150\,000$ queries per stratum, with
standard errors below $1.3\times10^{-3}$.

\paragraph{Networks.}
For $T\in\HH_B$ let $V(T)$ be the set of first fragments of the clusters
met by $T$. Then $V(T)\in\HH_B$, and the $1$-Lipschitz property of the maximum, used
as in the proof of Corollary~\ref{cor:noise}, gives
$\sup_q|h_T(q)-h_{V(T)}(q)|\le\eps$. The two-level network is
$\wh S\to V(T)\to T$, with $42$ coarse and $4483$ fine edges in every
replication; the flat network joins every competitor directly to $\wh S$. In
this design the envelope $\eps$ is known by construction; in applications it
would have to be certified externally (Proposition~\ref{prop:rare-obstruction}).

\paragraph{Certificates and error budgets.}
All certificates have total level $\alpha=0.05$. The weight correction of
Lemma~\ref{lem:propweights} uses $\alpha_\pi=\alpha_g=0.005$ and is common to
all of them; the remaining $0.04$ is allocated to the audit fluctuation. We
compare seven certificates of the form $\ov R=\min\{1,\max(0,\max_TU_T)\}$:
\begin{itemize}
\item[] (i) \emph{Hoeffding, flat}: a union of one-sided Hoeffding bounds over
the $HJ$ pairs $(h,T)$, which gives
$Z_T\le\sum_h\wh\pi_h\,2\{\log(HJ/0.04)/(2n_h)\}^{1/2}$;
\item[] (ii)--(iv) \emph{one-stage empirical Bernstein}
(Theorem~\ref{th:direct-empirical}, Corollary~\ref{cor:regret-one-stage}) with
$\alpha_a=0.04$: flat; two-level with $\alpha_1=\alpha_2=0.02$ and the known
envelope $\bar b_{h,e}=2\eps$ on fine edges; and two-level with the universal
envelope $\bar b_{h,e}=2$;
\item[] (v)--(vi) \emph{calibrated Bernstein}
(Lemma~\ref{lem:residual}, Theorem~\ref{th:empirical},
Theorem~\ref{th:regret}) with $\alpha_a=0.03$ and $\alpha_v=0.01$: flat, and
two-level with $\alpha_1=\alpha_2=0.015$. On fine edges, where
$|G_e|\le\eps$, the residual range is taken as $(\eps+|m_{h,e}|)^2$ and the
variance bound is truncated at $\eps^2$; both are direct consequences of
Lemma~\ref{lem:empbern} and of Popoviciu's inequality for a variable with
range $2\eps$;
\item[] (vii) \emph{pooled Bennett} (Proposition~\ref{prop:pooled-bennett})
with $\alpha_B=0.03$ and $\alpha_v'=0.01$. Each $\lambda_T$ minimizes, over a
fixed grid, the bound \eqref{eq:bennett-pooled} in which $v^+_{h,T}$ is
replaced by the same residual formula evaluated on the pilot with the known
audit counts; thus $\lambda_T$ uses no audit label.
\end{itemize}
We vary $n_h\in\{100,200,400,800,1600\}$ at $\eps=0.05$, and
$\eps\in\{0,0.02,0.05,0.1,0.2,0.4\}$ at $n_h=400$, with $200$ replications
per setting.

\subsection{Results}

\paragraph{Validity.}
In all $10$ settings and for all seven certificates, the certificate covered
the true regret in every replication, and the simultaneous audit event
$\{Z_T\le\text{radius}(T)\ \forall T\}$ held in every replication. With
$200$ replications, the one-sided $95\%$ Clopper--Pearson lower bound for each
coverage probability is $0.985$. As expected from union bounds with
conservative constants, the certificates are valid with room to spare.

\begin{table}[t]
\centering\small
\caption{Mean regret certificate $\ov R$ over 200 replications, $\eps=0.05$, as a function of the number $n_h$ of audited queries per stratum. Standard errors are below $0.003$ in all cells. The last rows give the mean true regret $R_B(\wh S)$ and the mean largest weight-correction term $\max_T e_\pi\sum_{h\ge1}\rho_{h,T}$, common to all certificates.}\label{tab:sim-n}
\begin{tabular}{@{}lccccc@{}}\hline
Certificate & $n_h=100$ & $n_h=200$ & $n_h=400$ & $n_h=800$ & $n_h=1600$\\ \hline
Hoeffding, flat & 0.550 & 0.381 & 0.265 & 0.187 & 0.134\\
One-stage EB, flat (Thm.~\ref{th:direct-empirical}) & 0.930 & 0.507 & 0.276 & 0.156 & 0.092\\
One-stage EB, two-level & 0.744 & 0.409 & 0.224 & 0.129 & 0.077\\
One-stage EB, two-level, $\bar b_{h,e}=2$ & 1.000 & 0.723 & 0.381 & 0.208 & 0.117\\
Calibrated Bernstein, flat (Thm.~\ref{th:empirical}) & 0.340 & 0.186 & 0.101 & 0.061 & 0.037\\
Calibrated Bernstein, two-level & 0.293 & 0.161 & 0.088 & 0.053 & 0.033\\
Pooled Bennett (Prop.~\ref{prop:pooled-bennett}) & 0.332 & 0.182 & 0.099 & 0.060 & 0.037\\
\hline
True regret & 0.0037 & 0.0027 & 0.0029 & 0.0035 & 0.0044\\
Weight correction & 0.033 & 0.028 & 0.019 & 0.014 & 0.012\\
\hline\end{tabular}
\end{table}

\paragraph{Size of the certificates.}
Table~\ref{tab:sim-n} and Figure~\ref{fig:sim-n} report the mean certificates
at $\eps=0.05$. The calibrated Bernstein certificate of
Theorem~\ref{th:regret} and the pooled Bennett certificate are nearly
indistinguishable, and both are about $2.5$--$2.7$ times smaller than the
one-stage empirical-Bernstein certificate at every $n_h$. The one-stage
certificate pays the linear term $7\bar b_{h,e}t/\{3(n_h-1)\}$ in every
stratum; for $n_h\le400$ it is even larger than the flat Hoeffding
certificate, which it overtakes only for $n_h\ge800$. In this design,
therefore, the separate calibration budget $\alpha_v$ is well spent, in line
with Remark~\ref{rem:no-dominance}. All certificates
remain far above the true regret, which is about $0.003$: at $n_h=1600$ the
smallest certificate is $0.033$, of which about $0.012$ is the
weight-correction term. This gap is the price of a simultaneous, honest
guarantee over $4525$ data-dependent competitors.

\begin{figure}[t]
\centering
\IfFileExists{fig-cert-vs-n.pdf}{\includegraphics[width=\textwidth]{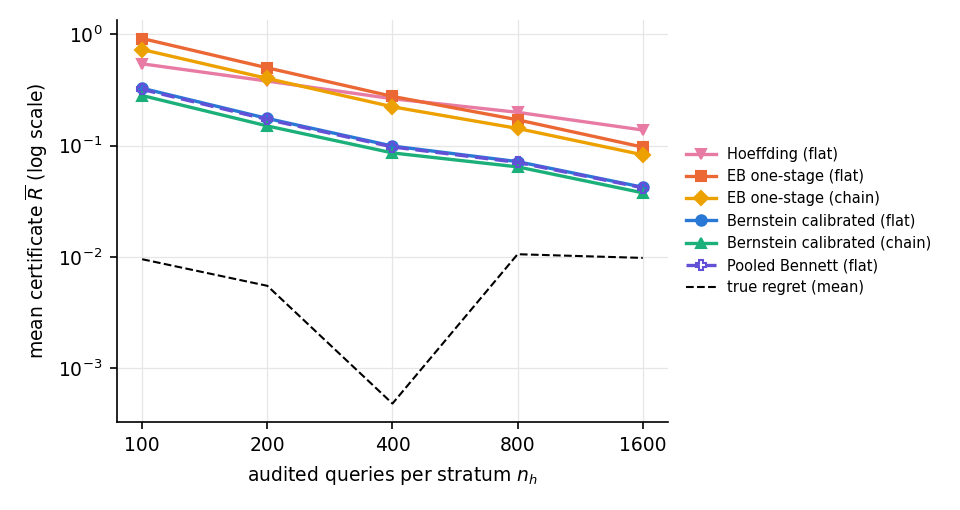}}{%
  \IfFileExists{fig-cert-vs-n.png}{\includegraphics[width=\textwidth]{fig-cert-vs-n}}{%
    \fbox{\parbox[c][35mm][c]{0.85\linewidth}{\centering Figure missing: \texttt{fig-cert-vs-n}\\ Add the original PDF or PNG before submission.}}%
  }%
}
\caption{Mean regret certificate $\ov R$ (logarithmic scales) over 200 replications at $\eps=0.05$, as a function of the number of audited queries per stratum. The pooled Bennett certificate (dashed, open marks) nearly coincides with the flat calibrated Bernstein certificate. The dashed black line is the mean true regret.}\label{fig:sim-n}
\end{figure}

\begin{table}[t]
\centering\small
\caption{Mean regret certificate $\ov R$ over 200 replications at $n_h=400$, as a function of the redundancy envelope $\eps$. Rows marked ``ratio'' divide the two-level certificate by the corresponding flat one.}\label{tab:sim-eps}
\begin{tabular}{@{}lcccccc@{}}\hline
Certificate & $\eps=0$ & $\eps=0.02$ & $\eps=0.05$ & $\eps=0.1$ & $\eps=0.2$ & $\eps=0.4$\\ \hline
Calibrated Bernstein, flat & 0.101 & 0.104 & 0.101 & 0.104 & 0.107 & 0.109\\
Calibrated Bernstein, two-level & 0.082 & 0.087 & 0.088 & 0.097 & 0.113 & 0.139\\
\quad ratio & 0.81 & 0.84 & 0.87 & 0.93 & 1.05 & 1.27\\ \hline
One-stage EB, flat & 0.278 & 0.280 & 0.276 & 0.278 & 0.280 & 0.276\\
One-stage EB, two-level & 0.209 & 0.218 & 0.224 & 0.243 & 0.277 & 0.337\\
One-stage EB, two-level, $\bar b_{h,e}=2$ & 0.374 & 0.380 & 0.381 & 0.391 & 0.409 & 0.436\\
\quad ratio (known envelope) & 0.75 & 0.78 & 0.81 & 0.87 & 0.99 & 1.22\\
\quad ratio ($\bar b_{h,e}=2$) & 1.35 & 1.36 & 1.38 & 1.41 & 1.46 & 1.58\\ \hline
Pooled Bennett, flat & 0.099 & 0.102 & 0.099 & 0.102 & 0.105 & 0.107\\
\hline\end{tabular}
\end{table}

\paragraph{When does the multiscale network help?}
Table~\ref{tab:sim-eps} and Figure~\ref{fig:sim-eps} vary the redundancy
envelope at $n_h=400$. With a known envelope, the two-level certificate is
smaller than the flat one for small $\eps$: by $19\%$ ($\eps=0$), $13\%$
($\eps=0.05$) and $7\%$ ($\eps=0.1$) for the calibrated Bernstein version, and
by $25\%$, $19\%$ and $13\%$ for the one-stage version. The advantage
vanishes between $\eps=0.1$ and $\eps=0.2$ for the calibrated version and at
about $\eps=0.2$ for the one-stage version, and it turns into a loss at
$\eps=0.4$. This crossover agrees with the deterministic threshold of
Corollary~\ref{cor:twolevel}, which equals about $0.19$ for $K_0=42$,
$J=4525$, $n_h=400$, $\alpha_1=\alpha_2$ and the present weights. The
reduction of the radius at the population optimum can be much larger than
the reduction of the certificate (from $0.061$ to $0.013$ at $\eps=0$ for the
calibrated version), because the certificate is a maximum over all
competitors. Finally, without a certified envelope ($\bar b_{h,e}=2$ on fine
edges) the two-level certificate is $35$--$58\%$ \emph{larger} than the flat
one: splitting the error budget across levels is then not compensated by any
reduction of the linear term. This illustrates why the gains of
Section~\ref{sec:gain} require certified, not merely empirical, envelopes
(Proposition~\ref{prop:rare-obstruction}).

\begin{figure}[t]
\centering
\IfFileExists{fig-ratio-vs-eps.pdf}{\includegraphics[width=0.85\textwidth]{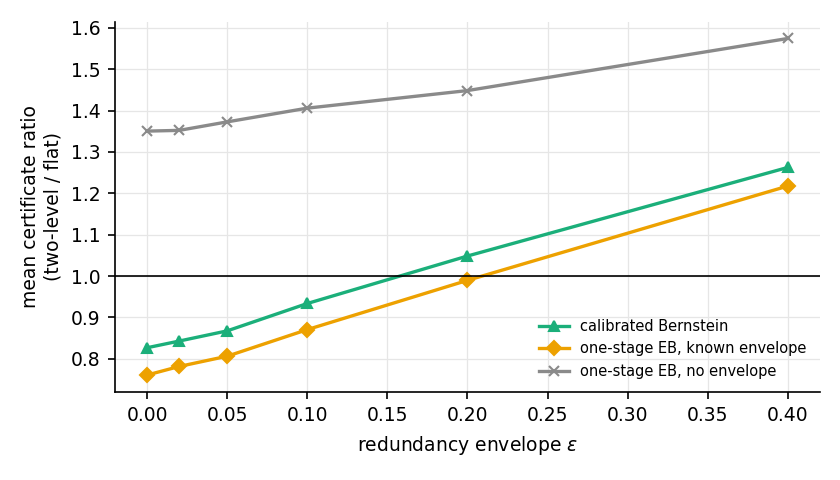}}{%
  \IfFileExists{fig-ratio-vs-eps.png}{\includegraphics[width=0.85\textwidth]{fig-ratio-vs-eps}}{%
    \fbox{\parbox[c][35mm][c]{0.85\linewidth}{\centering Figure missing: \texttt{fig-ratio-vs-eps}\\ Add the original PDF or PNG before submission.}}%
  }%
}
\caption{Ratio of the mean two-level certificate to the mean flat certificate at $n_h=400$, as a function of the redundancy envelope $\eps$; values below $1$ indicate a gain from the multiscale network.}\label{fig:sim-eps}
\end{figure}

\paragraph{Summary.}
All certificates are valid. Among the fully observable ones, variance
calibration (Theorem~\ref{th:regret}) or the pooled Bennett tilt
(Proposition~\ref{prop:pooled-bennett}) gives the smallest flat
certificates, and the multiscale network yields an additional reduction of
up to about $20\%$ when fragments are strongly redundant and the redundancy is
certified, in agreement with Corollary~\ref{cor:twolevel}.

%% =====================================================================
%\section{Discussion, limitations, and research direction}\label{sec:discussion}
%% =====================================================================

%\paragraph{What has been proved.}
We have specified a complete finite-sample probability argument:
pilot screening, conditional concentration of every pre-audit comparison
edge, telescoping over finite multiscale networks, simultaneous empirical
variance calibration, correction for unknown context proportions, and
conversion to postselection regret. The redundancy lemmas give explicit
variance contractions for swapping documentary fragments. The confidence
level in Theorem~\ref{th:regret} is unconditional: it averages over the
pilot, whose only contribution is the screening error $\alpha_s$, and
it requires no model for $P$. Conditionally on $\mathscr F_0$, only the
post-screening part of the statement holds, with level
$1-\alpha_a-\alpha_v-\alpha_\pi-\alpha_g$. %The procedure is honest, but its
%radius need not be small.
%Corollary~\ref{cor:regret-one-stage} provides a separate,
%fully observable benchmark that avoids the variance-calibration split. The
%simulations of Section~\ref{sec:simulation} indicate that, in a realistic
%redundant-fragment design, the calibrated certificate is markedly smaller than
%the one-stage one, and that the multiscale network adds a further reduction
%only when redundancy envelopes are small and certified.

%\paragraph{Sampling qualifications.}
%An acquired query is assumed to reveal all fragment relevances needed
%for the retained contrasts; if only selected fragment--query pairs are
%labeled, the missing labels require a separate design. Likewise,
%selection with probabilities depending on unmodeled features within
%strata invalidates Assumption~\ref{ass:audit}. Sequential stopping on
%observed labels is outside this paper's fixed-audit guarantee and would
%require time-uniform methods (see, e.g., \citealp{HowardEtAl21}).
%The effective cost is in labeled \emph{fragment--query pairs}, not
%merely in acquired query counts. A small audit radius obtained after a
%large, fully labeled pilot is not necessarily cost-effective overall.

%\paragraph{Geometry and the limits of observability.}
The sets $\mathcal N_\ell$ may be constructed using an independent pilot,
including estimated similarities. Such a construction does not need to
be accurate for the base theorem: validity follows from the edgewise
simultaneous concentration with conservative envelopes.% However, a


\begin{thebibliography}{99}
%% =====================================================================

\bibitem[Bartlett et~al.(2005)]{BartlettBousquetMendelson05}
Bartlett, P.~L., Bousquet, O., and Mendelson, S. (2005).
Local Rademacher complexities. \emph{The Annals of Statistics},
33(4), 1497--1537. \href{https://doi.org/10.1214/009053605000000282}{doi:10.1214/009053605000000282}.

\bibitem[Boucheron et~al.(2013)]{BoucheronLugosiMassart13}
Boucheron, S., Lugosi, G., and Massart, P. (2013).
\emph{Concentration Inequalities: A Nonasymptotic Theory of Independence}.
Oxford University Press.
\href{https://doi.org/10.1093/acprof:oso/9780199535255.001.0001}{doi:10.1093/acprof:oso/9780199535255.001.0001}.

\bibitem[Burgess and Chapman(2021)]{BurgessChapman21}
Burgess, M.~A. and Chapman, A.~C. (2021).
Approximating the Shapley value using stratified empirical Bernstein sampling.
In \emph{Proceedings of the 30th International Joint Conference on Artificial Intelligence},
pp.~73--81.
\href{https://doi.org/10.24963/ijcai.2021/11}{doi:10.24963/ijcai.2021/11}.

\bibitem[Dirksen(2015)]{Dirksen15}
Dirksen, S. (2015). Tail bounds via generic chaining.
\emph{Electronic Journal of Probability}, 20(53), 1--29.
\href{https://doi.org/10.1214/EJP.v20-3760}{doi:10.1214/EJP.v20-3760}.

\bibitem[Dud\'{\i}k et~al.(2011)]{DudikLangfordLi11}
Dud\'{\i}k, M., Langford, J., and Li, L. (2011).
Doubly robust policy evaluation and learning.
In \emph{Proceedings of the 28th International Conference on Machine Learning},
pp.~1097--1104. \url{https://icml.cc/2011/papers/554_icmlpaper.pdf}.

\bibitem[Hoeffding(1963)]{Hoeffding63}
Hoeffding, W. (1963). Probability inequalities for sums of bounded random
variables. \emph{Journal of the American Statistical Association},
58(301), 13--30.
\href{https://doi.org/10.1080/01621459.1963.10500830}{doi:10.1080/01621459.1963.10500830}.

\bibitem[Howard et~al.(2021)]{HowardEtAl21}
Howard, S.~R., Ramdas, A., McAuliffe, J., and Sekhon, J. (2021).
Time-uniform, nonparametric, nonasymptotic confidence sequences.
\emph{The Annals of Statistics}, 49(2), 1055--1080.
\href{https://doi.org/10.1214/20-AOS1991}{doi:10.1214/20-AOS1991}.

\bibitem[Koltchinskii(2006)]{Koltchinskii06}
Koltchinskii, V. (2006). Local Rademacher complexities and oracle
inequalities in risk minimization. \emph{The Annals of Statistics},
34(6), 2593--2656.
\href{https://doi.org/10.1214/009053606000001019}{doi:10.1214/009053606000001019}.

\bibitem[Krause and Golovin(2014)]{KrauseGolovin14}
Krause, A. and Golovin, D. (2014). Submodular function maximization.
In Bordeaux, L., Hamadi, Y., and Kohli, P. (eds.), \emph{Tractability:
Practical Approaches to Hard Problems}, Cambridge University Press,
pp.~71--104. \href{https://doi.org/10.1017/CBO9781139177801.004}{doi:10.1017/CBO9781139177801.004}.

\bibitem[Maurer and Pontil(2009)]{MaurerPontil09}
Maurer, A. and Pontil, M. (2009).
Empirical Bernstein bounds and sample variance penalization.
In \emph{Proceedings of the 22nd Annual Conference on Learning Theory (COLT)}.
\url{https://arxiv.org/abs/0907.3740}.

\bibitem[Nemhauser et~al.(1978)]{NemhauserWolseyFisher78}
Nemhauser, G.~L., Wolsey, L.~A., and Fisher, M.~L. (1978).
An analysis of approximations for maximizing submodular set functions---I.
\emph{Mathematical Programming}, 14, 265--294.
\href{https://doi.org/10.1007/BF01588971}{doi:10.1007/BF01588971}.

\bibitem[Talagrand(2005)]{Talagrand05}
Talagrand, M. (2005). \emph{The Generic Chaining: Upper and Lower Bounds of Stochastic Processes}.
Springer. \href{https://doi.org/10.1007/3-540-27499-5}{doi:10.1007/3-540-27499-5}.

\end{thebibliography}
\end{document}